\documentclass[11pt]{amsart}
\usepackage[margin=1in,marginparwidth=0.8in, marginparsep=0.1in]{geometry}
\usepackage{graphicx}
\usepackage{amssymb}
\usepackage{amsthm}
\usepackage{amsfonts}
\usepackage{amscd}
\usepackage{mathabx}
\usepackage{hyperref}

\numberwithin{equation}{section}

\newcommand{\C}{\mathbb{C}}     
\newcommand{\R}{\mathbb{R}}     

\DeclareMathOperator{\Ho}{H}
\DeclareMathOperator{\baar}{B}

\newtheorem{theorem}{Theorem}[section]            
\newtheorem{lemma}[theorem]{Lemma}               
\newtheorem{corollary}[theorem]{Corollary}              
\theoremstyle{remark}              
\theoremstyle{remark}           

\begin{document}

\title{Legendrian surgery}
\author{Tobias Ekholm}
\address{Department of mathematics and Center for Geometry and Physics, Uppsala University, Box 480, 751 06 Uppsala, Sweden and
Institut Mittag-Leffler, Aurav 17, 182 60 Djursholm, Sweden}
\email{tobias.ekholm@math.uu.se}

\thanks{TE is supported by the Knut and Alice Wallenberg Foundation, KAW2020.0307 Wallenberg Scholar and by the Swedish Research Council, VR 2022-06593, Centre of Excellence in Geometry and Physics at Uppsala University and VR 2020-04535, project grant.}

\begin{abstract}	
This is an overview paper that describes Eliashberg's Legendrian surgery approach to wrapped Floer cohomology and use it to derive the basic relations between various holomorphic curve theories with additional algebraic constructions. We also give a brief discussion of further results that use the surgery perspective, e.g., for holomorphic curve invariants of singular Legendrians and Lagrangians.     
\end{abstract}

\maketitle

\tableofcontents

\section{Introduction}\label{Sec Intro}
Starting some time around 2004, Yasha Eliashberg initiated the handle approach to computing holomorphic curve invariants of Weinstein manifolds and their contact boundaries, focusing in particular on relating the most elementary part of closed string symplectic field theory (SFT) \cite{EGH}, orbit contact homology, to the simplest part of open string SFT, the Legendrian- or Chekanov-Eliashberg dg-algebra, of the Legendrian attaching spheres for critical handles. Although perhaps not the original intention, this lead to connections and isomorphisms between all the various flavours of `tree level' holomorphic curve theories for Weinstein manifolds. Corresponding results for `higher loops' constitute very interesting problems where some initial progress has been made but much remains to be understood. 

In this paper we will explain the steps in the proof of the most important and basic Legendrian surgery isomorphism that expresses the wrapped Floer cohomology of the co-core disks of a Weinstein manifold in terms of the Chekanov-Eliashberg dg-algebra of attaching spheres. We will then discuss extensions and ramifications of this result to related theories: symplectic cohomology with its product structure, Hochschild homology and cohomology of wrapped Floer cohomology, partially wrapped Fukaya categories and Chekanov-Eliashberg dg-algebras with coefficients in chains on the based loop space coefficients, associated dg-algebras for singular Legendrians, and cut-and-paste methods.

\section{Weinstein manifolds}
Consider an exact symplectic $2n$-manifold $X$ with symplectic form $\omega=d\lambda$. A vector field $Z$ which is $\omega$-dual to $\lambda$, $\lambda=\iota_{Z}\omega$ is called a \emph{Liouville vector field}. For such a vector field, if $L$ denotes the Lie derivative, $L_{Z}\omega=d(\iota_{Z}\omega) + \iota_{Z}d\omega=\omega$,  which means that $\omega$ expands along the flow of $Z$ in the positive direction and contracts in the negative. 

The manifold $X$ is a \emph{Weinstein manifold} if $Z$ is complete and if it admits a Morse function $F\colon X\to\R$ for which $Z$ is gradient like. It follows in particular that outside a compact subset, $X$ is symplectomorphic to $Y\times [0,\infty)$ where $Y=F^{-1}(T)$ for some sufficiently large $T$ and where the symplectic form on $Y\times[0,\infty)$ is $d (e^{t}\alpha)$, $\alpha=\lambda|_{Y}$. Here $(Y,\alpha)$ is a contact manifold which is the \emph{ideal contact boundary} of $X$. Compact versions, $\overline X=X\setminus F^{-1}(T,\infty)$ of $X$ are called \emph{Weinstein domains}, one can move between Weinstein domains and Weinstein manifolds by adding or removing cylindrical ends.   

The most basic example of a Weinstein manifold is $\C^{n}\approx\R^{2n}$, the associated Weinstein domain is the $2n$-ball $B^{2n}$, with the standard symplectic form $\sum_{j=1}^{n} dx_{j}\wedge dy_{j}= d (xdy)$, Liouville vector field $\sum_{j} \frac12\left(x_{j}\partial_{x_{j}} + y_{j}\partial_{y_{j}}\right)$, and ideal contact boundary the standard contact sphere $S^{2n-1}$ with contact structure given by complex tangencies. 

Other examples are cotangent bundles $T^{\ast} M$ of compact manifolds $M$ with Liouville form $pdq$. Equip $M$ with a Riemannian metric take the Liouville vector field as $\sum_{j} \frac12 p_{j}\partial_{p_{j}}+X_{H}$, where $H=pdq(\nabla f)$ for some small Morse function $f\colon M\to\R$, $X_{H}$ is the Hamiltonian vector field, and the exhausting Morse function is $F(q,p)=\frac12 p^{2}+f(q)$.

Consider now a Weinstein manifold $X$. Using the equation $L_{Z}\omega=\omega$, it is straightforward to check that the stable manifolds $W^{\mathrm{s}}$ of the critical points are isotropic, i.e., $\omega$ vanishes along $W^{\mathrm{s}}$. In particular, critical points have indices $\le n$, and ordering the critical levels according to index then gives the following handle decomposition of $X$: 
\begin{align*}
&\overline{X} \ = \ \overline{X}^n \ \supset \ \overline{X}^{n-1} \ \supset \ \dots \ \supset \ \overline{X}^{1} \ \supset  \ \overline{X}^0, \text{ where }\\ 
&\overline{X}^{m} \ = \ \overline{X}^{m-1} \ \cup_{\Lambda^{m}} \ H^{m}, \quad m=1,\dots,n.
\end{align*}
Here $\overline{X}^{0}$ is the $2n$-ball and $\overline{X}^{m}$ is obtained from $\overline{X}^{m-1}$ by attaching $m$-handles $H^{m}\approx \bigsqcup \overline{T^{\ast} D^{m}}\times B^{2(n-m)}$ along a collection of isotropic $(m-1)$-spheres $\Lambda^{m}$ (the descending spheres of the isotropic stable manifolds) in the contact boundary $Y^{m-1}$ of $X^{m-1}$. 

For $m<n$ there is an $h$-principle for isotropic $m$-spheres in contact $(2n-1)$-manifolds, which means that the symplectic topology of $X^{m}$ is uniquely determined by the homotopy class of the tangent map of $\Lambda_{m}$. Consequently all the interesting symplectic topology of $X=X^{n}$ is carried by the Legendrian isotopy class of the attaching spheres $\Lambda=\Lambda^{n}\subset Y^{n-1}=\partial \overline{X}^{n-1}$. This observation can be taken as the starting point for the Legendrian surgery approach to symplectic topology invariants. It says that all symplectic invariants of $X$ can be obtained from $\Lambda$. 

We will call the stable manifolds $L$ of the index $n$ critical points the \emph{core disks} and the unstable manifolds $C$ its \emph{co-core disks}. The first Legendrian surgery result we will explain can then be stated as saying that the wrapped Floer cohomology $CW^{\ast}(C)$ of $C$ is quasi-isomorphic to the Chekaonv-Eliashberg dg-algebra $CE^{\ast}(\Lambda)$ of $\Lambda$. 

Below we will simplify notation and write $\overline{X}=\overline{X}_{0}\cup_{\Lambda} H$, where $\overline{X}_{0}$ is the \emph{subcritical} part of $X$, i.e., the sublevel set of all handles of index $<n$ and $H$ is the union of \emph{critical} $n$-handles, one copy of $T^{\ast}D^{n}$ for each handle, attached along $\Lambda\subset Y_{0}=\partial\overline{X}_{0}$. We will also simplify notation and write e.g., $X=X_{0}\cup_{\Lambda} H$, etc, and not distinguish in notation between a Weinstein domain and the Weinstein manifold which is its completion. We use the notation $Y=\partial_{\infty} X$ and $Y_{0}=\partial_{\infty} X_{0}$ to denote ideal contact boundaries. Furthermore, we will write $L$ and $C$ for the stable and unstable manifolds of the index $n$ critical points in $H$, and call them the Lagrangian core and co-core disks, respectively. See Figure \ref{fig:handles}.
\begin{figure}[htbp]
   \centering
   \includegraphics[width=.5\linewidth]{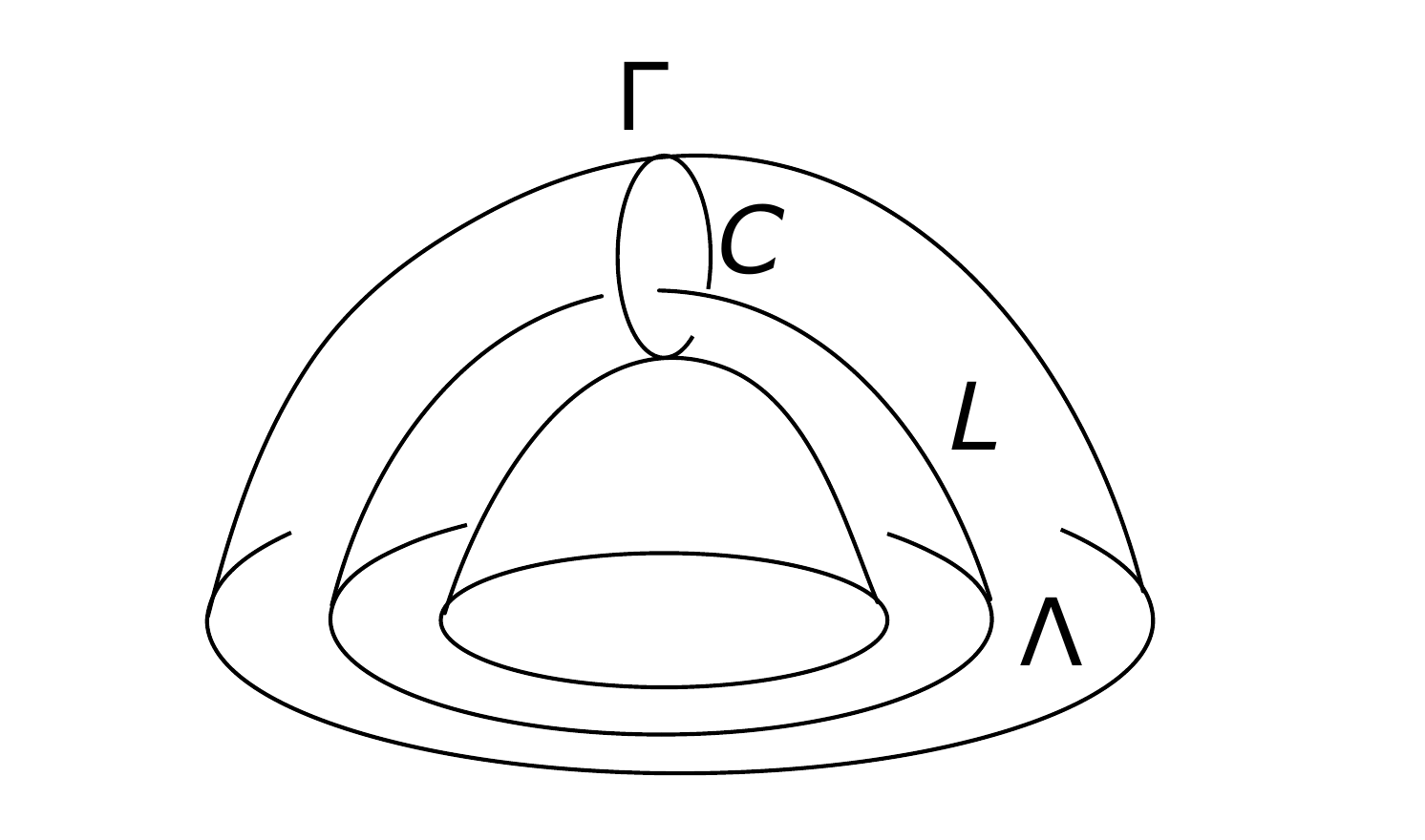} 
   \caption{A Lagrangian handle, core disk $L$ and attaching sphere $\Lambda$, co-core disk $C$ with Legendrian boundary $\Gamma$.}
   \label{fig:handles}
\end{figure}

\section{The basic surgery isomorphism}
In this section we give the basic steps in the Legendrian surgery isomorphism connecting the wrapped Floer cohomology of Lagrangian co-core disks to Chekanov-Eliashberg dg-algebras of Legendrian attaching spheres.

\subsection{Reeb orbits and Reeb chords}
Consider a contact manifold $V$ with contact form $\alpha$, in our case $V=Y=\partial_{\infty} X$  or $V=Y_{0}=\partial_{\infty} X_{0}$ and $\alpha$ is the restriction of $\lambda$. The \emph{Reeb vector field} $R$ of $\alpha$ is the vector field that generates the kernel of $d\alpha$, $\iota_{R}d\alpha=0$, normalized by $\alpha(R)=1$. 

Flow loops of $R$ are called \emph{Reeb orbits}. For generic contact form, Reeb orbits are isolated and transverse (linearized return map in general position). For our treatment of holomorphic curve theories we will use parameterized Reeb orbits. This means that we fix a point on each Reeb orbit and consider all its multiples parameterized, starting the parameterization at that point. If $\Lambda\subset V$ is a Legendrian submanifold then \emph{Reeb chords} of $\Lambda$ are flow lines of $R$ starting and ending on $\Lambda$, again for generic data Reeb chords are isolated and transverse (linearized flow from start point to endpoint in general position). 

We next introduce the chain complexes underlying our main holomorphic curve theories. We start with the theory after handle attachment.

\subsubsection{Generators of the wrapped Floer cohomology complex}\label{gen CW}
Consider a Lagrangian $C\subset X$ with ideal Legendrian boundary $\Gamma\subset Y$. Below $C$ will be the co-core disks of a Lagrangian handle attachment, but the definition makes sense more generally. The wrapped Floer cohomology complex of $C$ is most often defined by picking a Hamiltonian displacement of $C$ with prescribed behavior at infinity and taking generators Hamiltonian chords of $C$ and differential counting perturbed holomorphic curves asymptotic to the generators. For Legendrian surgery purposes it is necessary to use the following equivalent definition. 

Consider a Morse function $f\colon C\to C$ with gradient equal to the Liouville vector field at infinity. Let $C_{-}$ be a time $-\epsilon$ shift along $f$, viewed as a Hamiltonian in $T^{\ast} C$. Then the wrapped Floer cohomology complex $CW^{\ast}(C)$ of $C$ is generated by Reeb chords starting on $C$ and ending on $C_{-}$ and intersection points in $C\cap C_{-}$. Since the shift is small we see that generators correspond to Reeb chords of $\Gamma$ and critical points of $f$, see \cite[Appendix B.1]{EL}. 

\subsubsection{Generators of the Chekanov-Eliashberg dg-algebra}\label{gen CE}
Before handle attachment we use the Chekanov-Eliashberg dg-algebra $CE^{\ast}(\Lambda)$ of a Legendrian $\Lambda\subset Y_{0}$. It is the algebra of composable words of Reeb chords of $\Lambda$, where a word of chords is composable if the Reeb chord endpoint of a chord lies in the same connected component of $\Lambda$ as the start point of the following chord. Multiplication is concatenation, where the product is zero if the result is non-composable.

\subsection{Handle attachement and Reeb dynamics}
The basis of Legendrian surgery isomorphisms is the following `homoclinic orbit' observation.
\begin{lemma}\label{l:newchords}
Reeb chords of of the co-core boundary sphere $\Gamma=\partial_{\infty} C$ are in natural 1-1 correspondence with composable words of Reeb chords of the Legendrian attaching link $\Lambda$. More precisely, given an action level $\mathfrak{a}>0$ there exists a handle size $\delta>0$ such that for all handles of size samller than $\delta$, Reeb chords of $\Gamma$ of action less than $\mathfrak{a}$ are in natural 1-1 correspondence with composable words of Reeb chords of $\Lambda$ of total action less than $\mathfrak{a}$. 
\end{lemma}

\begin{proof}
Outside the handle attachment region the Reeb flows in $Y_{0}$ and $Y$ are naturally identified. We take the co-core as the fiber of $T^{\ast} D$ at the center of $D$. Inside the handle the Reeb flow is the lift of the geodesic flow. In particular, the Reeb flow takes fiber hemispheres of the attaching neighborhood to the boundary of the attaching region by a degree one map. A straightforward fixed point argument now establishes the 1-1 correspondence, see \cite[Section 5]{Ecurves} for details.  
\end{proof}

\subsection{Wrapped Floer cohomology}
We first consider the differential $d$ in the wrapped Floer cohomology complex $CW^{\ast}(C)$. This differential counts anchored holomorhic strips with a positive and negative puncture at a Reeb chord connecting $\Gamma_{-}$ to $\Gamma$ or an intersection point between $C_{-}$ and $C$. Here anchoring means that we consider disks with additional negative punctures at marked orbits and pure chords where we fill with rigid punctured spheres and disks. The basic property of the curve counting map is the following.

\begin{lemma}
	The map $d\colon CW^{\ast}(C)\to CW^{\ast}(C)$ is a differential $d^2=0$.
\end{lemma}
   
\begin{proof}
The square of the differential counts the ends of a 1-dimensional moduli space, see \cite[Appendix B.1]{EL}.
\end{proof}   

In fact, the differential is the $\mu_{1}$-operation in an $A_{\infty}$-algebra structure on $CW^{\ast}(C)$, where the operation $\mu_{k}$ is defined by choosing $k$ parallel copies of $C$, where $C_{-}$ is the first parallel copy which is the time $-1$ shift along the function $\epsilon f_{1}$ for $f_{1}=f$ and the $k^{\rm th}$ copy is the time $-1$ shift along the function  
\[
\left(\sum_{r=1}^{k} \epsilon^{k}\right) f_{1} + \left(\sum_{r=2}^{k} \epsilon^{3+r}\right)f_{2} + \dots +
\epsilon^{k+1}f_{k}.
\] 

Then the $\mu_{k}$-operation counts disks with positive punctures connecting copies in increasing order and one negative output puncture, see Figure \ref{fig:CWdiff}. 
\begin{figure}[htbp]
   \centering
   \includegraphics[width=.6\linewidth]{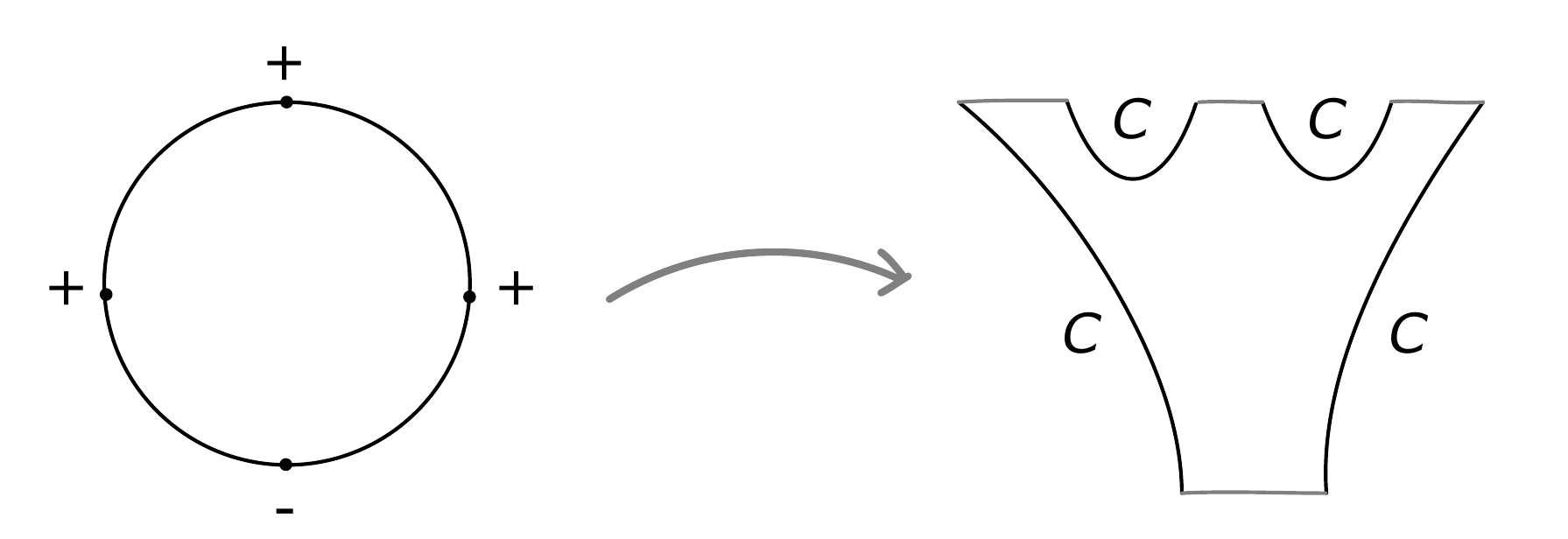}
   \caption{Holomorphic disks for the $\mu_k$-operations on $CW^{\ast}(C)$.}
   \label{fig:CWdiff}
\end{figure}

We then have the following result.

\begin{lemma}
The operations $\mu_{k}\colon CW^{\ast}(C)^{\otimes k}\to CW^{\ast}(C)$ satisfy the $A_{\infty}$-relations.
\end{lemma} 

\begin{proof}
Again the compositions count the ends of a 1-dimensional moduli space. The proof uses that for sufficiently close parallel copies, moduli spaces with boundary components on Lagrangians corresponding to any increasing numbering in the system of parallel copies are be canonically identified. This follows from the form of the shifting functions and the fact that transversely cut out solutions persist under sufficiently small perturbation, see \cite[Appendix B.1]{EL}.
\end{proof}

\subsection{The Chekanov-Eliashberg dg-algebra}
We next consider the differential $\partial$ of the Chekanov-Eliashberg dg-algera. The differential counts holomorphic disks in the symplectization $Y_{0}\times \R$ anchored at orbits with one positive and several negative punctures, see Figure \ref{fig:CEdiff}. 
\begin{figure}[htbp]
   \centering
   \includegraphics[width=.6\linewidth]{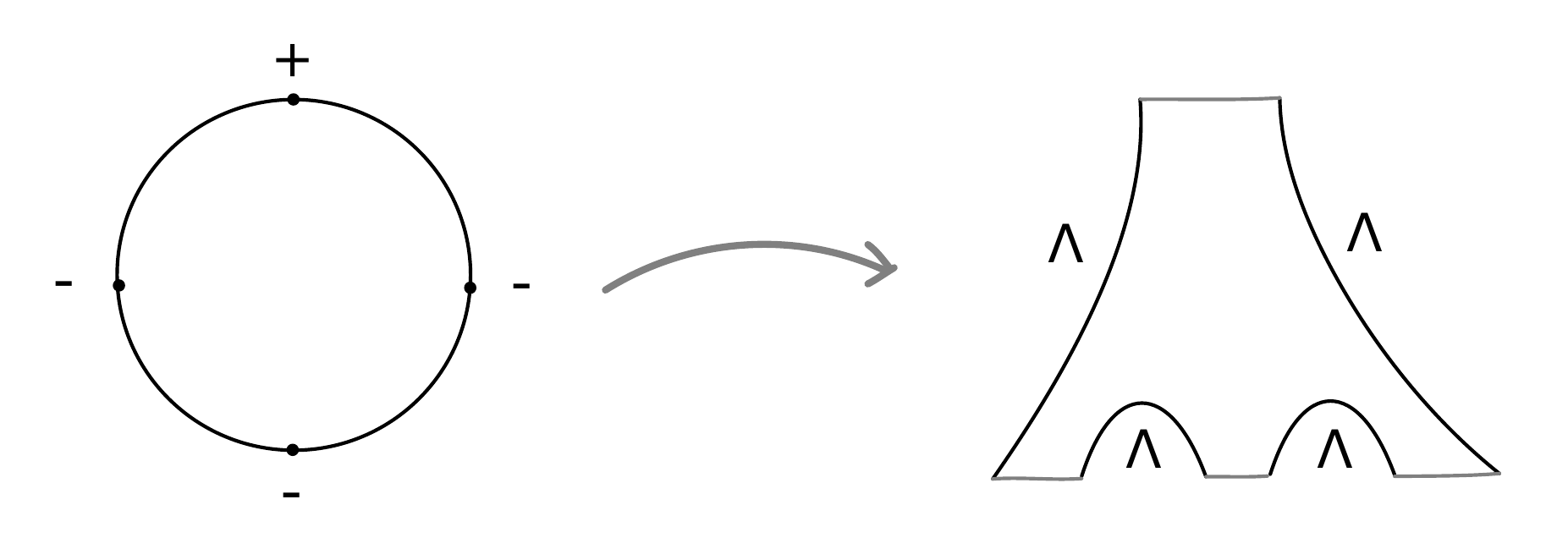}
   \caption{Holomorphic disks for the differential on $CE^{\ast}(\Lambda)$.}
   \label{fig:CEdiff}
\end{figure}

We extend it to a dg-algebra differential by the Leibniz rule.

\begin{lemma}
The map $\partial\colon CE^{\ast}(\Lambda)\to CE^{\ast}(\Lambda)$ is a differential, $\partial^{2}=0$.
\end{lemma}

\begin{proof}
The square of the differential counts ends of a 1-dimensional moduli space, see e.g., \cite{EL}.
\end{proof}

\subsection{The chain map}
The cobordism $W=X\setminus X_{0}$, with the Lagrangians $L$ and $C$ with connected components that intersect at single points in the middle of each handle, gives a natural chain map 
\begin{equation}\label{eq:CW to CE}
\Phi^{CW}\colon CW^{\ast}(C)\to CE^{\ast}(\Lambda),\quad 
\Phi^{CW}=\sum_{k=1}^{\infty}\Phi^{CW}_{j},
\end{equation}
More precisley, $\Phi_{k}^{CW}$ counts holomorphic curves with $k$ positive punctures at Reeb chords of $C$, two punctures where the map is asymptotic to intersections in $C\cap L$ and several negative punctures at Reeb chords of $\Lambda$, where moduli spaces of many positive punctures of $C$ are defined using systems of parallel copies as in the definition of the $A_{\infty}$-operations in order to avoid `boundary breaking', see Figure \ref{fig:CWtoCE}.
\begin{figure}[htbp]
   \centering
   \includegraphics[width=.6\linewidth]{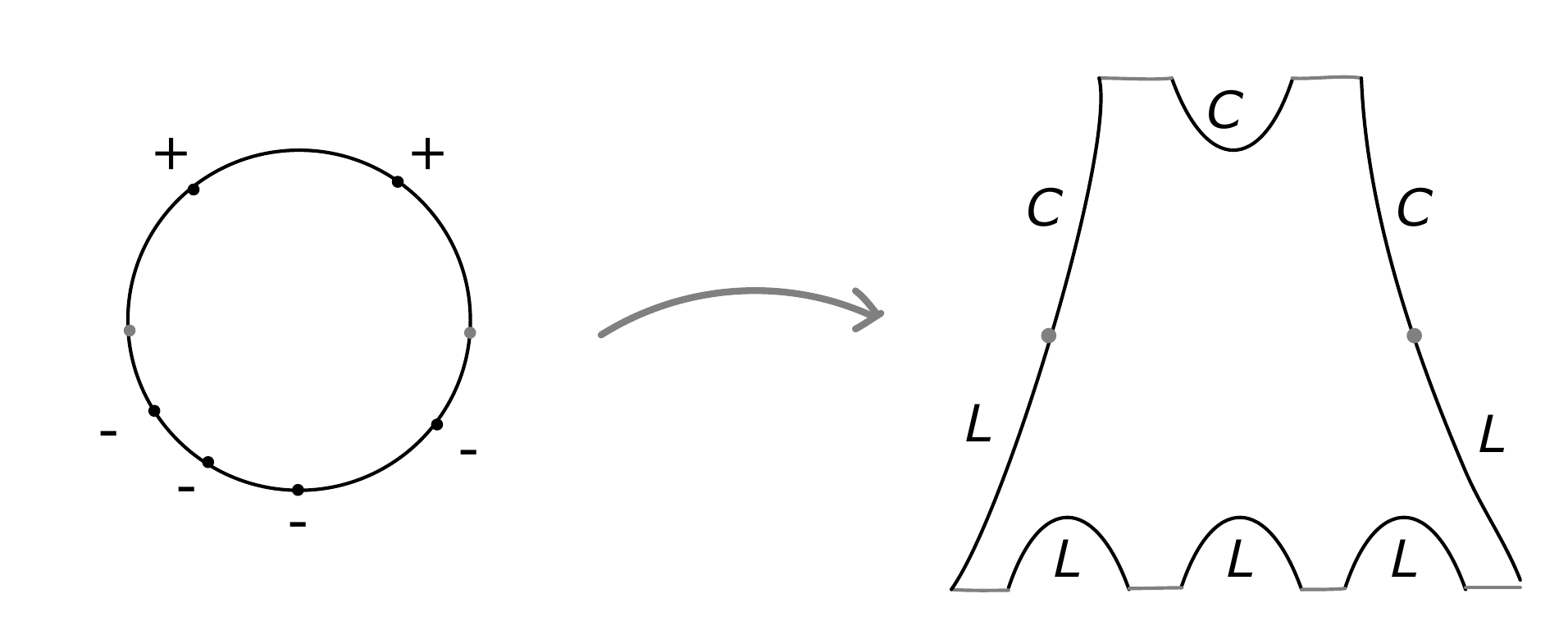}
   \caption{Holomorphic disks for the surgery map $CW^{\ast}(C)\to CE^{\ast}(\Lambda)$.}
   \label{fig:CWtoCE}
\end{figure}
We view $CE^{\ast}(\Lambda)$ as an $A_{\infty}$-algebra generated by words of Reeb chords with differential $\mu_{1}=\partial$, with $\mu_{2}$ given by the concatenation product, and with all higher $\mu_k$ equal to zero.

\begin{lemma}
The map $\Phi^{CW}=\sum_{k}\Phi_{k}^{CW}$ in \eqref{eq:CW to CE} is an $A_{\infty}$-map.
\end{lemma} 

\begin{proof}
To see this we consider $1$-dimensional moduli spaces of disks as in the definition of $\Phi_{k}$ and note that the boundary points of such a moduli does correspond to splittings at Reeb chords, which is precomposing $\Phi_{k}$ with $\mu_{j}$ in the positive end or post-composing $\Phi_{k}$ with $\mu_{1}$ at the negative end, or splitting at one of the Lagrangian intersection points in $C\cap L$, which corresponds to post-composing $\Phi_{j}$ with $\mu_{2}$ in the negative end. The lemma follows, see \cite[Appendix B.2]{EL} for more details. 
\end{proof}

\subsection{Chain isomorphism}
We show in this section that the $A_{\infty}$-map $\Phi^{CW}$ in \eqref{eq:CW to CE} is in fact a chain isomorphism. More precisely, we have the following.
\begin{lemma}
For any action cut-off $\mathfrak{a}_{0}>0$, the image under $\Phi_{1}^{CW}$ of a Reeb chord $\overline{w}$ of $\Gamma$ corresponding to a word $w$ of Reeb chords of $\Lambda$ of action $\mathfrak{a}(\overline w)<\mathfrak{a}_{0}$  is
\[
\Phi_{1}^{CW}(\overline{w}) \ = \ \pm w  \ + \ E(\overline{w}),
\]
where $E(\overline{w})$ is a sum of Reeb chord words $v$ of action $\mathfrak{a}(v)< \mathfrak{a}(w)$. 
\end{lemma}

\begin{proof}
The proof uses induction on action. For one letter words an explicit construction gives one transversely cut out strip which is unique by an action argument. Assume inductively that the count of curves in moduli space of disks connecting $\overline{w}$ to $w$ for all words of length $\le m$ equals $\pm 1$. Then consider the moduli space of disks with two positive punctures at $\overline{w}_{0}$ and $\overline{w}_{1}$ and negative punctures at $w_{0}w_{1}$. By action this moduli space has only two boundary breakings, see Figure \ref{fig:oneone}: breaking at a point in $C\cap L$, by induction we know that the count of such configurations equal $\pm 1$, and breaking into a disk with two positive and one negative puncture followed by an isomorphism disk, we conclude that both of these moduli spaces must also contain $\pm 1$ elements, see \cite[Section 7]{Ecurves}.  
\end{proof}

\begin{figure}[htbp]
   \centering
   \includegraphics[width=.5\linewidth]{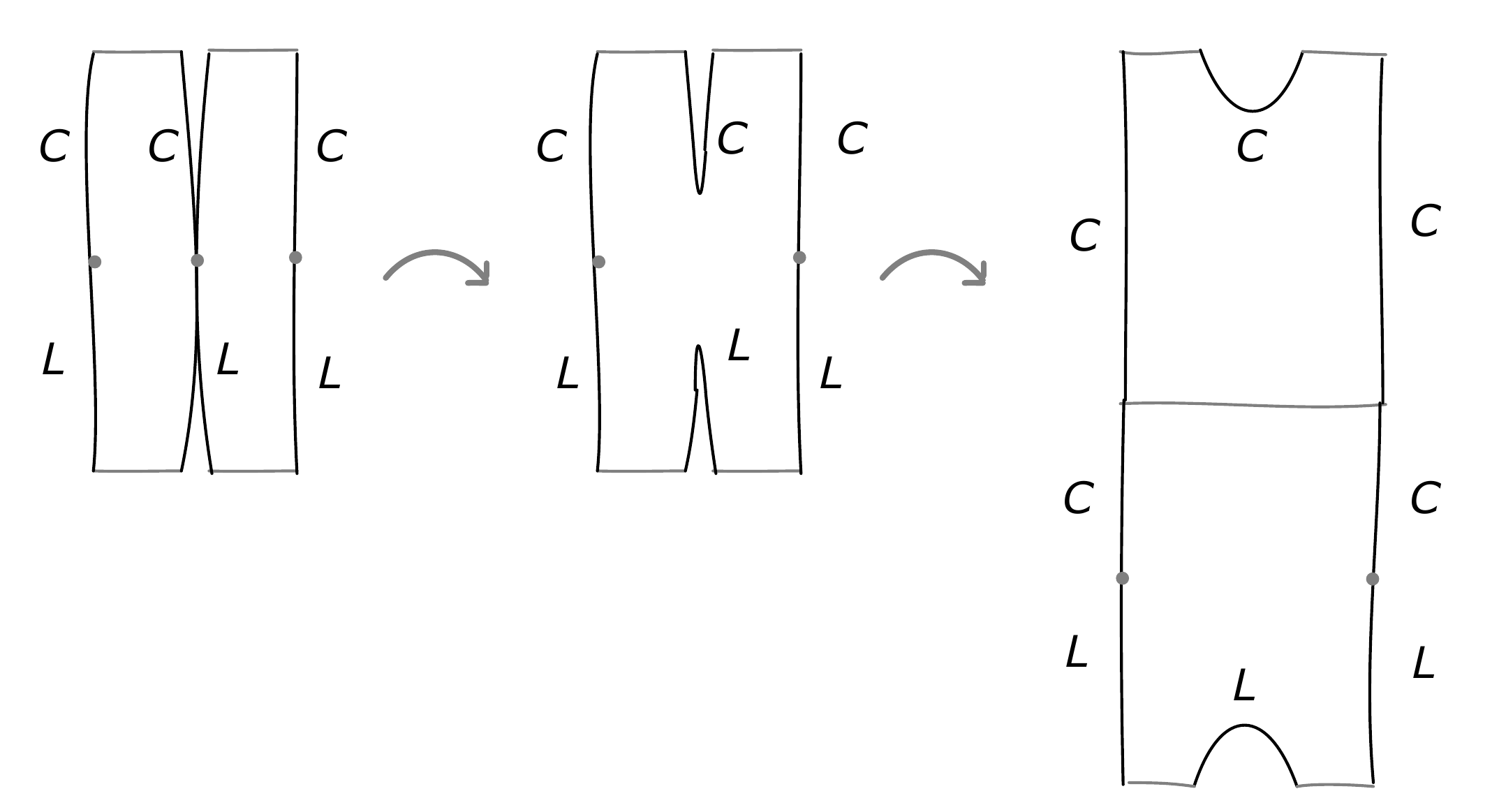}
   \caption{Constructing isomorphism disks by gluing.}
   \label{fig:oneone}
\end{figure}

Using a straightforward action filtration argument, see \cite[Appendix B.2]{EL}, we then get the main surgery isomorphism theorem:
\begin{theorem}\label{t:CW to CE}
The natural $A_{\infty}$-chain map 
\[
\Phi^{CW}\colon CW^{\ast}(C)\to CE^{\ast}(\Lambda)
\]
is a quasi-isomorphism.\qed
\end{theorem}

\section{Further results}
The Legendrian surgery proof is powerful and can be used to derive many closely related results. In this section we will discuss some of them.

\subsection{Symplectic homology and linearized contact homology}
Changing the surgery isomorphism we can instead map into symplectic cohomology. The standard construction of symplectic cohomology of $X$ starts from a time dependent Hamiltonian with standard linear behavior at infinity of slope not in the action spectrum of the contact form. The complex is generated by Hamiltonian orbits and the differential counts solutions to the Cauchy-Riemann equations perturbed by the Hamiltonian vector field on cylinders that interpolate between orbits. Here we will use a version of symplectic cohomology that arises by first making the Hamiltonian time independent, where time independent orbits is viewed as Bott degenerate time dependent orbits, and then turning the Hamiltonian off, where perturbed curves become standard holomorphic curves with asymptotic markers remebering the Morse theory on the orbits. 

Our complex for symplectic homology is thus Reeb orbits together with a Morse complex of the underlying geometric orbit and critical points of a Morse function on $X$. The differential counts anchored holomorhic cylinders and spheres with Morse data at orbits, we call these decorated orbits. We denote the corresponding complex $SC^{\ast}(X)$, see \cite[Section 3.3]{BEE}.     

We next want to model this complex before surgery. The counterpart of Lemma \ref{l:newchords} says that Reeb orbits after surgery are in natural 1-1 correspondence with Reeb orbits before surgery and cyclic words of Reeb chords of the attaching spheres, see \cite[Section 5]{Ecurves}. In order to model this situation we use a two copy $L_0\cup L_{1}$ of the descending manifold. Here $L_{0}\cap L_{1}=\{z\}$ is one point near the middle of the handle and Reeb chords of the boundary of $\Lambda_0\cup\Lambda_1$ consists of two copies of Reeb chords of $\Lambda$ and two additional Reeb chords $x$ and $y$ for each component of $\Lambda$ corresponding to the maximum and a minimum of a Morse function making the Reeb shift generic. 

We call a Reeb chord connecting $\Lambda_0$ to $\Lambda_1$ \emph{mixed} and chords connecting $\Lambda_{j}$ to itself \emph{pure}. We then consider the two copy Chekanov-Eliashberg dg-algebra complex which is generated by words of chords in which exactly one chord is mixed and others are pure. The differential counts holomorphic curves with mixed positive puncture, one mixed negative puncture, and any number of pure negative punctures. Rather than using this complex alone we use the (very small) Lagrangian Floer cohomology complex of $L_0$ and $L_{1}$ with coefficients in the two copy complex, where the differential counts also curves with one Lagrangian intersection point and negative end with one mixed puncture. In this case, this means simply add the generator $z$ and observe that $\partial z = y$. We denote this complex $CF^{\ast}((L_{0},\Lambda_0),(L_{1},\Lambda_1))$ and write simply $\Ho CF^{\ast}(L,\Lambda)$ for the corresponding Hochschild complex that is obtained by identifying words up to cyclic permutation.

There is now a natural surgery cobordism map 
\[
\Phi^{SC}=\Phi_{\mathrm{o}}^{SC}\oplus\Phi_{\mathrm{w}}^{SC}\colon SC^{\ast}(X) \ \to \ SC^{\ast}(X_0)\oplus \Ho CF^{\ast}(L,\Lambda),
\]
where $\Phi_{\mathrm{o}}^{SC}$ is the standard cobordism map counting cylinders between Morse decorated orbits, and where $\Phi_{\mathrm{w}}^{SC}$ counts disks with positive puncture at a decorated orbit, mixed distinguished puncture at $1$ and a puncture mapping to $z$ at $-1$. We think of the right hand side as a complex with differential
\[
d = \left(\begin{matrix}
	d_{\mathrm{oo}} & d_{\mathrm{wo}}\\
	d_{\mathrm{ow}} & d_{\mathrm{ww}}
\end{matrix}\right),
\]      
where $d_{\mathrm{oo}}$ is the usual cylinder counting differential on $SC^{\ast}(X_0)$, $d_{\mathrm{wo}}=0$, $d_{\mathrm{ww}}$ is induced from the differential on $\Ho CF^{\ast}(L,\Lambda)$, and where $d_{\mathrm{ow}}$ counts disks with positive puncture at a Morse decorated orbit and distinguished negative puncture at mixed chord, similar to $\Phi_{\mathrm{w}}$. 

\begin{theorem}\label{t:SC to HCF}
The map $\Phi^{SC}$ is a chain isomorphism.
\end{theorem}

\begin{proof}
	We adapt the Morse functions on the orbits corresponding to cyclic words so that these functions have a maximum and a minimum on each chord on the underlying orbit. It then follows that $\Phi^{CW}$ takes an orbit with maximum on a chord $c$ to the corresponding word of chords with the chord $c$ mixed, and the orbit with a minimum on $c$ to the corresponding word with $c$ replaced by $cx$, where $c$ is the pure chord and $x$ is the mixed chord at the minimum Reeb chord of the shift. The proof is then directly analogous to Theorem \ref{t:CW to CE}.  
\end{proof}

This has the following consequence.
\begin{corollary}\label{c:SC iso HCF}
The chain map $\Phi_{\mathrm{w}}^{SC}\colon SC^{\ast}(X) \to  \Ho CF^{\ast}(L,\Lambda)$ is a quasi isomorphism.
\end{corollary}
\begin{proof}
The complex $SC^{\ast}(X_0)$ is contractible as the symplectic homology complex of a subcritical manifold.
\end{proof}

Combining Theorem \ref{t:CW to CE} and Corollary \ref{c:SC iso HCF} we learn that the Hochschild complex $\Ho CW^{\ast}(C)$ of $CW^{\ast}(C)$ is quasi-isomorphic to the symplectic homology complex $SC^{\ast}(X)$. 

It is also possible to define the counterpart of $SC^{\ast}(X)$ without Morse data on the orbits. The corresponding complex is known as cylindrical contact homology and is isomorphic to the $S^{1}$-equivariant version of $SC^{\ast}(X)$. The standard approach to defining this $S^{1}$-equivariant version is to define a BV-operator that deforms the Hamiltonian perturbation by rotating the domain. This operation does not square to zero and one has to add (infinitely many) correction terms. However, in the Morse-Bott description of $SC^{\ast}(X)$ the BV-operator $\xi$ admits a simple description that does square to zero: if $\gamma$ is a Reeb orbit, and $\widehat{\gamma}$ and $\widecheck{\gamma}$ 
denote $\gamma$ decorated by a maximum and by a minimum, respectively, then $\xi(\widecheck{\gamma})=\widehat{\gamma}$ and $\xi(\widehat{\gamma})=0$. The corresponding operation on cyclic words is $\underline{x}w\mapsto \sum \underline{w}$, where the sum goes over all ways of choosing mixed chord in $w$. With this explicit form of the BV-operator it is straightforward to obtain the cylindrical contact homology using a model for $S^{1}$-equivariant homology on $\Ho CF^{\ast}(L,\Lambda)$ together with the $S^{1}$-action given by the BV-operator $\xi$ satifying $\xi^{2}=0$.  

\subsection{Upside down surgery}
A Lagrangian handle has the form $T^{\ast}D^{n}$. Here we think of $\partial D^{n}\times\{0\}$ as the attaching sphere $\Lambda$ that we fill by the core disk $D^{n}\times \{0\}$ which is then $L$. Note that the dynamics of the handle is the same if we view it from the other side: start from the boundary of the fiber at $0$, i.e., $\Gamma$ the boundary of the co-core disk, and attach the co-core disk $C$. We call this up-side down Lagrangian handle attachment or upside-down Legendrian surgery. This then means that if we shrink the size of the handle around $\Gamma$ we have the counterpart of Lemma \ref{l:newchords}:
\begin{lemma}
	If $Y_{0}$ is eqipped by the contact form induced from that on $Y$ by upside-down Legendrian surgery, Reeb chords of $\Lambda$ are in natural 1-1 correspondence with composable words of Reeb chords of $\Gamma$. \qed
\end{lemma}  

By analogy with $\Phi^{CW}$, we now have a chain map of infinity co-algebras
\[
\Phi_{CW}\colon \baar CW^{\ast}(C)\to LCE^{\ast}(\Lambda).
\]
Here the left hand side is the bar complex of $CW^{\ast}(C)$, i.e., $\baar CW^{\ast}(C)$ is generated by words of Reeb chords of $\Gamma$ and critical points in $C$ with differential induced by the $\mu_{k}$-operations. This bar complex $\baar CW^{\ast}(C)$ is naturally a (infinity) co-algebra with differential $c_1$, co-product $c_2$ corresponding to splitting words in two all possible ways and all higher $c_{k}$ equal to zero. The right hand side is the linearized Chekanov-Eliashberg dg-algebra (here we assume that $CE^{\ast}(\Lambda)$ has an augmentaion) where the operations $c_{k}$ correspond to the part of the (augmented) differential with $k$ negative punctures. In direct analogy with Theorem \ref{t:CW to CE} we have the following.
\begin{theorem}\label{t: CW to LCE}
The map $\Phi_{CW}$ is a quasi-isomorphism.\qed
\end{theorem} 

We consider also the cyclic version of this map. We start from the Hochschild complex $\Ho CW^{\ast}(C)$ of $CW^{\ast}(C)$ generated by cyclic words of Reeb chords one of which is distinguished. In analogy with the core disk, we think of this as a two copy complex of where the distinguished puncture is mixed and where we also have the intersection point $z=C_{0}\cap C_1$ as a generator with $\partial z=x$, where is the minimum of the Reeb shift at infinity and where the maximum $y$ of the Reeb shift plays a role analogous to $x$ for the two copy of $L$. We then add to $\Ho CW^{\ast}(C)$ the complex $SC^{\ast}(X)$. The differential counts, except for the usual curves also disks with several positive punctures and a negative puncture at the center, the location of the distinguished positive puncture at a fixed boundary location determined by the marker at orbit in the center and an anti-podal puncture mapping to $z$. We denote this complex 
\[
\Ho CW^{\ast}(C)\oplus SC^{\ast}(X)
\] 
and in analogy with Theorem \ref{t:SC to HCF} we have a natural chain map
\[
\Psi\colon \Ho CW^{\ast}(C)\oplus SC^{\ast}(X) \ \to \ SC^{\ast}(X_{0})
\]
which is a chain isomorhism. The connecting homomorphism in the long exact sequence associated to the short exact sequence  
\[
0 \ \to \ SC^{\ast}(X) \ \to \ \Ho CW^{\ast}(C)\oplus SC^{\ast}(X) \ \to \ \Ho CW^{\ast}(C) \ \to \ 0 
\]
is called the \emph{open-closed map}, see \cite{Ab,G}. We will denote it 
\[
\mathcal{OC}\colon \Ho CW^{\ast}(C)\to SC^{\ast}(X),
\]
see Figure \ref{fig:OC}.
\begin{figure}[htbp]
   \centering
   \includegraphics[width=.6\linewidth]{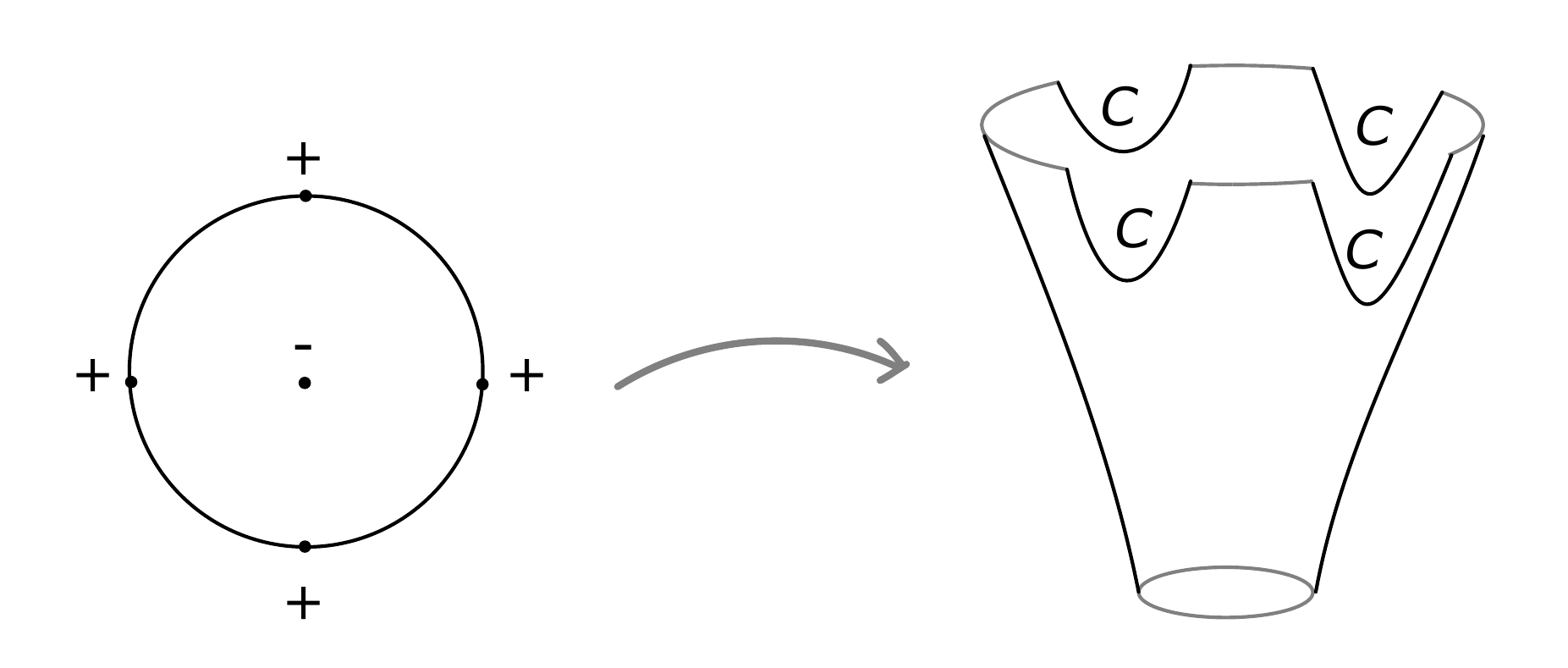} 
   \caption{Curves counted by the open-closed map.}
   \label{fig:OC}
\end{figure}

Since $SC^{\ast}(X_{0})$ is contractible we find:
\begin{corollary}
The open-closed map $\mathcal{OC}\colon \Ho CW^{\ast}(C)\to SC^{\ast}(X)$ is a quasi-isomorphism.\qed 
\end{corollary}

\subsection{The closed-open map and isomorphisms of Hochschild homology and cohomology}
The open-closed and closed-open maps were originally studied in the setting of wrapped Floer cohomology from the Hamiltonian view point, see \cite{G,Ab}. Here we continue instead with the Legendrian surgery perspective and use wrapped Floer cohomology without Hamiltonian, see \cite[Appendix B.1]{EL}. 

The Hochschild chain complex $\Ho CW^{\ast}(C)$ above can be thought of as generated by words of negative punctures at Reeb chords along the boundary of a formal disk (one of which is distinguished, we take it to be mixed). The Hochschild differential then attaches holomorphic disks with several positive and one negative puncture to such words in all possible ways (the negative and one of the positive punctures mixed). Similarly, we consider the Hochschild cochain complex $\Ho' CW^{\ast}(C)$ generated by chords along the boundary of a formal disk, all positive except one which is negative (as in the differential the negative and one positive puncture mixed). The differential on $\Ho' CW^{\ast}(C)$ is obtained by attaching the disks corresponding to the differential on $\Ho CW^{\ast}(C)$ with one negative and several positive punctures in all possible ways, from above and below. 

The complex $\Ho' CW^{\ast}(C)$ has a natural product which acts by gluing negative and positive punctures of one to the other. This product has a natural unit $e$ which is the sum of words of two punctured disks with the same chord at the positive and negative puncture. Furthermore, there is a curve counting map, \emph{the closed-open map}  
\[
\mathcal{CO}\colon SC^{\ast}(X)\to \Ho' CW^{\ast}(C) 
\]
which counts disks with a positive puncture with marker at the center mapping to an orbit with Morse decoration, one mixed negative boundary puncture at a fixed position, any number of positive boundary punctures with the mixed chord at the point opposite the fixed position, see Figure \ref{fig:CO}. 

\begin{figure}[htbp]
   \centering
   \includegraphics[width=.6\linewidth]{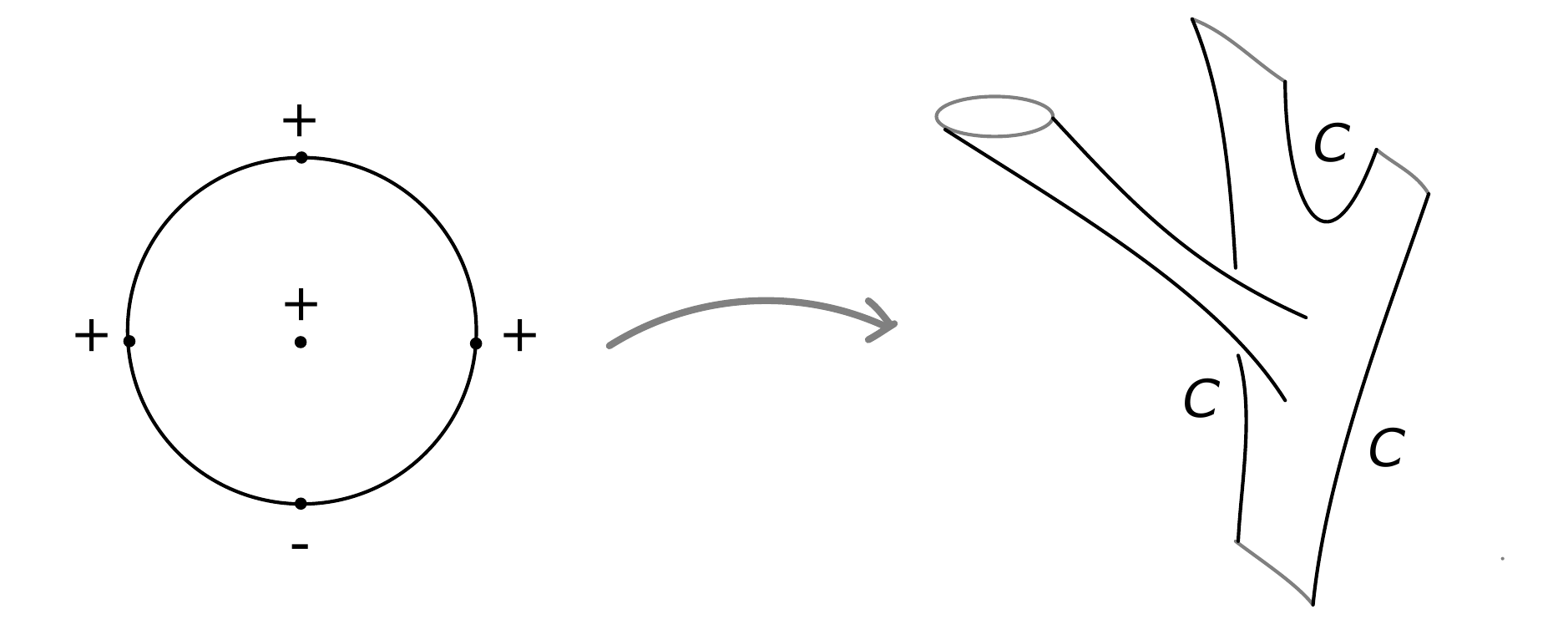} 
   \caption{Curves counted by the closed-open map}
   \label{fig:CO}
\end{figure}

\begin{lemma}
The map $\mathcal{CO}$ is a chain map that respects the product (on homology).
\end{lemma}   
\begin{proof}
The symplectic cohomology product followed by the isomorphism gives a disk with two positive punctures. Looking at possible splittings we find the product on the Hochschild cochains up to exact terms.	
\end{proof}

Consider now the complex $\Ho CW^{\ast}(C)$. This is naturally a $\Ho'CW^{\ast}(C)$-module: if $a\in \Ho CW^{\ast}(C)$ is a cyclic word and $r\in\Ho' CW^{\ast}(C)$ then $r\cdot a$ is the sum of cyclic words obtained by attaching the positive punctures of $r$ to consecutive (negative) punctures of $a$. Precomposing the product by the map $\mathcal{CO}$ we find that $\Ho CW^{\ast}(C)$ is also an $SC^{\ast}(X)$-module. The symplectic homology complex has the pairs of pants product and is hence an $SC^{\ast}(X)$-module itself. We have the following.

\begin{lemma}
The open-closed map $\mathcal{OC}\colon \Ho CW^{\ast}(C)\to SC^{\ast}(X)$ is a map of $SC^{\ast}(X)$-modules.
\end{lemma}

\begin{proof}
The boundary of the 1-dimensional (reduced) moduli spaces of curves with positive chord punctures at the boundary, one negative and one positive interior puncture on the distinguished ray from the center to the distinguished puncture  correspond exactly to first multiplying and then applying the isomorphism or first applying the isomorphism and then multiplying. 
\end{proof}

Together with a degeneration of moduli spaces argument this leads to the following result. 

\begin{theorem}\label{t: OC iso}
The map $\mathcal{CO}\colon SC^{\ast}(X)\to \Ho' CW^{\ast}(C)$ is a quasi-isomorphism.
\end{theorem}
\begin{proof}
We use two parallel fibers $C_0$ and $C_1$ as source of the map $\mathcal{OC}$ and target of the map $\mathcal{CO}$, respectively. We will take them to lie very close together. Note that for the Legendrians at infinity of the parallel fibers, the shift from $\partial_{\infty}C_0$ to $\partial_{\infty} C_{1}$ is induced by a Morse function on the sphere that is the restriction of a linear function that is negative at the south pole, vanishes on the equator, and positive at the north pole. Note then that Reeb chords $C_0\to C_1$ and $C_1\to C_0$ correspond to Reeb chords $C\to C$ and short Reeb chords near $C_0$, and that therefore, there is a natural 1-1 correspondence between Reeb chords $C_0\to C_{1}$ and Reeb chords $C_1\to C_0$.  	
	
We remark that when we discuss moduli spaces of holomorphic curves for the Lagrangians $C_0$ and $C_{1}$ below, we need to use systems of parallel copies as described above. We use separate systems for $C_{0}$ and $C_{1}$, see \cite[Appendix B]{EL}, but will leave these systems implicit in the notation.

Let $1$ denote the unit in $SC^{\ast}(X)$. In the complex $SC^\ast(X)$, $1$ is represented by the minimum of the Morse function on $X$. Since $\mathcal{CO}$ is an isomorphism we find $u\in \Ho CF^{\ast}(C)$ such that $\mathcal{OC}(u)=1$. Take $s\in SC^{\ast}(X)$ then since $\mathcal{OC}$ is a map of $SC^{\ast}(X)$-modules with the module structure on $\Ho CF^{\ast}(C)$ induced by $\mathcal{CO}$ followed by the natural action of $\Ho' CF^{\ast}(C)$ on $\Ho CF^{\ast}(C)$, we find 
$$
\mathcal{OC} ( \mathcal{CO} (s) \cdot u )= s\cdot \mathcal{OC}(u)=s\cdot 1=s.
$$ 
It follows that $\mathcal{CO}$ is injective.

To show that $\mathcal{CO}$ is surjective, note that $\mathcal{CO}(1)$ counts curves with a positive puncture at the minimum on $X$. The only such curves correspond to flow lines starting at the minimum and ending at a fixed point (the midpoint say) of any Reeb chords. It follows that 
\[
\mathcal{CO}(1) \ = \ e \ =\sum_{\text{{\tiny Reeb chords $c$}}} c_{\mathrm{pos}}\otimes c_{\mathrm{neg}}.
\] 
 
 Consider the moduli spaces $\mathcal{M}(u;e)$ involved in the equation 
\[
\mathcal{CO}\circ\mathcal{OC}(u) = e.
\]
The elements in these moduli spaces are infinite length cylinders with positive punctures corresponding to $u$ at one boundary component and one positive and one negative puncture along the other boundary component corresponding to any chord in $e$. Gluing at the middle orbit we gain one dimension since the marker in the middle disappears. This means that the resulting moduli space has dimension two. 

Consider a cycle $r\in \Ho' CF^{\ast}(C)$ and a non-trivial contribution to $r\cdot u$. We use this contribution to degenerate the moduli spaces in $\mathcal{M}(u;e)$. We degenerate the annulus into two disks $D_{\mathrm{up}}$ with positive punctures according to the positive punctures in $r$ and the positive puncture in the $e$-term corresponding to the negative puncture of $r$, and two negative punctures, and a disk $D_{\mathrm{dn}}$ in the lower part with positive punctures at the chords in $r\cdot u$ that are not the negative chord of $r$. Under this deformation the moduli space undergoes additional splittings that will only affect the result up to exact terms, using the assumptions that $r$ and $u$ are cycles.     

We next glue $D_{\mathrm{up}}$ and $D_{\mathrm{dn}}$ via a cobordism corresponding to an isotopy that interchanges $\partial C_0$ and $\partial C_1$. The Lagrangian cobordism $T\subset \R\times\partial X$ corresponds to the Legendrian isotopy in $\partial X$ that rotates $\partial C_0$ to $\partial C_1$, i.e., and then takes the shifting function to its negative. It is easy to see that if $C_0$ and $C_1$ are sufficiently close and the rotation is sufficiently slow, all rigid holomorphic curves in the cobordism corresponds to (reparameterized) trivial Reeb chord strips of $c$ and strips (corresponding to Morse flow lines) that interchanges the small Reeb chords that follows the rotation. Let $\mathcal{M}(T)$ denote the moduli space of rigid holomorphic curves with boundary on $T$. 

Noting that the curves in $\mathcal{M}(T)$ are strips that exactly give the natural 1-1 correspondence between Reeb chords $C_0\to C_1$ and $C_1\to C_0$, we find that the result of gluing $D_{\mathrm{dn}}$ via $\mathcal{M}(T)$ to $D_{\mathrm{up}}$ are new annuli of dimension $2$ that have positive punctures according to $r\cdot u$ along one boundary component and positive and one negative puncture according to $r$. Then deforming the Lagrangians to the standard $C_0\cup C_1$ and following the domains until it splits at a Reeb orbit, we find that  
\[
\mathcal{CO}(\mathcal{OC}(r\cdot u))=r,
\]     
up to exact terms, see Figure \ref{fig:cardy}. It follows in particular that $\mathcal{CO}$ is surjective. The theorem follows. 
\end{proof}

\begin{figure}[htbp]
   \centering
   \includegraphics[width=.6\linewidth]{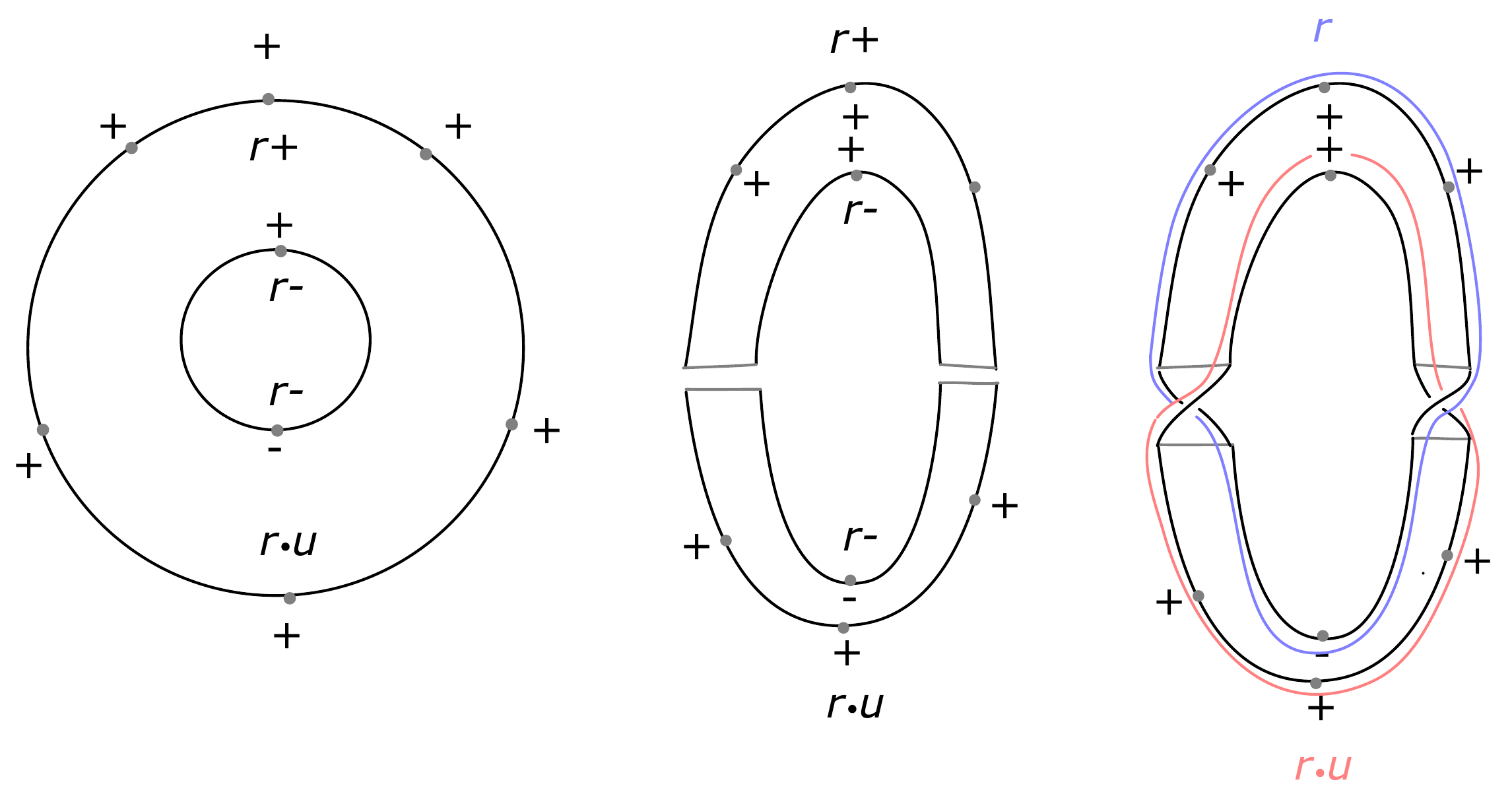} 
   \caption{Degeneration of the composition $\mathcal{CO}\circ\mathcal{OC}(u)$ according to $r\in \Ho'CW^{\ast}(C)$ and gluing by the rotation cobordism.}
   \label{fig:cardy}
\end{figure}

\subsection{The symplectic homology product and Calabi-Yau structures}
In this section we consider the product on $SC^{\ast}(X)$. The isomorphism $\Phi^{SC}\colon SC^{\ast}(X)\to \Ho CE^{\ast}(\Lambda)$ allows us to push the pair of pants product $\mu$ on $SC^{\ast}(X)$ to the Hochschild complex $\Ho CE^{\ast}(\Lambda)$. This is related to a duality Calabi-Yau structure on $CE^{\ast}(\Lambda)$. 

In order to compute the product we consider the composition 
\[
\Phi^{SC}\circ \mu\colon SC^{\ast}(X)\otimes SC^{\ast}(X)\to \Ho CE^{\ast}(\Lambda).
\] 
The curves contributing to this composition are disks with two positive punctures near the origin and negative boundary punctures and one extra puncture opposite the distinguished mixed puncture mapping to $z$. As in the proof of Theorem \ref{t: OC iso}, at the gluing we gain one dimension from a disappearing conformal constraint. We use it to control the conformal structure of the domain. Studying splittings we find the following, see \cite{BEE2}.

\begin{theorem}
The symplectic cohomology product is represented on $\Ho CE^{\ast}(\Lambda)$ by attaching disks with two positive punctures and auxiliary negative punctures to pairs of cyclic words of chords. Here the disks are of two types: disks with one Lagrangian intersection puncture at $z$ and one positive and one negative puncture at the same Reeb chord $c$ for any $c$, and two level disks: a triangle followed by a disk with two positive punctures, see \cite[Figure 1]{BEE2}.\qed   
\end{theorem}

\subsection{Partial wrapping and coefficients in chains of the based loop space}
Starting from the study of attachments of punctured handles along knot conormals in \cite{ENS}, a more systematic study of Legendrian surgery for partially wrapped Floer cohomology was undertaken in \cite{EL} where it was conjectured that the partially wrapped Floer cohomology of a linking disk at a Legendrian stop is isomorphic to the Chekanov-Elishberg dg-algebra of the Legendrian stop with coefficients in chains on its based loop space. This result was later proved \cite[Theorem 1.2]{AE} by translating everything to the surgery language. 

This translation shows that any construction in exact Floer theory admits a surgery description and gives, for example,  a rather direct and geometric proof of the cut and paste methods of \cite{GPS}, see \cite[Theorem 1.3]{AE} and \cite[Theorem 1.1]{A}. Here we will illustrate the idea by giving the definition of the surgery description of partially wrapped Floer cohomology and a sample result that can be proved with this technology.

Consider Weinstein domains $W$ of dimension $2n$ and $V$ of dimension $2n-2$. We fix a Lagrangian skeleton $L$ with a fixed handle structure for $V$ and think of $V$ as a cotangent bundle of $L$. We take a Legendrian embedding of $L$ into $\partial W$ to be a contact embedding of a neighborhood $V\times (-\epsilon,\epsilon)$ of the zero section in the contactization of $V$. Given such an embedding, consider $W$ with the fixed embedding of $V\times (-\epsilon,\epsilon) $ in $\partial W$ and also the negative half $[0,-\infty)\times\R\times V$ of the symplectization of $\R\times V$ with $V\times (-\epsilon,\epsilon)\times V$ embedded in $\{0\}\times \R\times V$. We attach a $V$-handle, $T^{\ast} I\times V$ to the union. Note that the handles of $V$ give attaching spheres for $X$. We define the Chekanov-Eliahberg dg-algebra of $L$ to be the Chekanov-Eliashberg dg-algebra of the Legendrian attaching spheres of $X$. Note that it is generated by Reeb chords of $L$ together with Reeb chords of the Legendrian attaching spheres in $V$ in the middle of the handle. 

In case $L$ is smooth, observe that the internal Reeb chords of $L$ give a dg-algebra quasi-isomorphic to the chains on the based loop space of $L$. Using this observation one can show that the dg-algebra of $L$ then is isomorphic to the dg-algebra with loop space coefficients, see \cite[Theorem 1.2]{AE}.

We give one illustration of how this construction can be used. Consider a smooth manifold $M$ and fix a Morse function $f\colon M\to\R$ with one maximum and one minimum. It is well known that the wrapped Floer cohomology of the cotangent fiber in $T^{\ast} M$ is quasi-isomorphic to chains on the based loop space of $M$, 
\[
CW^{\ast}(T_{m}M)\approx C_{\ast}(\Omega M).
\]
For the purposes of this discussion, we think of this result as not cutting $M$ at all, or to comply with what comes next, as cutting $M$ below the minimum of the Morse function. We can also cut $M$ just below the maximum. If $M_{\le a}$ is the sublevel set then the cotangent bundle $T^{\ast}M_{\le a}$, with corners rounded is a subcritical Weinstein domain with the level sphere $\Lambda_{a}=\partial M_{\le a}$ as a Legendrian submanifold. By the surgery isomorphism, Theorem \ref{t:CW to CE}, $CE^{\ast}(\Lambda_{a})$ is also isomorphic to $CW^{\ast}(T_{m} M)$. It turns out the same holds true if we cut at any regular level $b$. More precisely, we get a subcritical Weinstein domain $T^{\ast}M_{\le b}$ with a Legendrian level surface $\Lambda_{b}=f^{-1}(b)$. Consider now the Chekanov-Eliashberg dg-algebra of $\Lambda_{b}$ with coefficients in chains of the based loop space of the super-level set $M_{\ge b}$, 
\[
CE^{\ast}(\Lambda_{b};C_{\ast}(\Omega(M_{\ge b}))).
\]       
Then, in fact, 
\begin{equation}\label{eq:any reg}
CE^{\ast}(\Lambda_{b};C_{\ast}(\Omega(M_{\ge b}))) \ \approx \ CW^{\ast}(T_{m}^{\ast}M)
\end{equation}
for every regular value $b$.
The isomorphisms in \eqref{eq:any reg} illustrate the fact that in the exact case, Floer cohomology can be expressed either in terms of  topology or Reeb dynamics and further that  it is possible to interpolate between these extreme cases by having a part in topology and another in Reeb dynamics. 

The isomorphisms in this example are straightforward to derive working in Legendrian surgery presentations. We describe them here as an illustration of the great generality and flexibility of the Eliashberg's original surgery approach to holomorphic curve theories in Weinstein domains. It can handle very general objects and provides e.g., Chekanov-Eliashberg dg-algebra or Floer cohomology for Legendrians with boundary with corners with coefficients in chains on the based loop space, in terms of the usual dg-algebras of Legendrian attaching spheres in suitable Weinstein manifolds. Such singular Lagrangians and Legendrians are very useful for instance as centers of mirror symmetry charts see \cite{D-RET} for a simple example.


\begin{thebibliography}{999}
\bibitem{Ab}
M. Abouzaid, 
\emph{A geometric criterion for generating the Fukaya category}
Publ. Math. Inst. Hautes Études Sci. (2010), no. 112, 191--240.

\bibitem{A}
J. Asplund,
\emph{Simplicial descent for Chekanov-Eliashberg dg-algebras},
J. Topol. {\bf 16} (2023), 489--541


\bibitem{AE}
	J. Asplund, T. Ekholm,
	\emph{Chekanov-Eliashberg dg-algebras for singular Legendrians},
	J. Symplectic Geom. {\bf 20} (2022) 509--559.

\bibitem{BEE}
F. Bourgeois, T. Ekholm, Y. Eliashberg,
\emph{Effect of Legendrian surgery},
Geom. Topol. {\bf 16} (2012) 301--389.	

\bibitem{BEE2}
F. Bourgeois, T. Ekholm, Y. Eliashberg,
\emph{Symplectic homology product via Legendrian surgery},
Proc. Natl. Acad. Sci. USA {\bf 108} (2011) 8114--8121.

\bibitem{D-RET}
G. Dimitroglou Rizell, T. Ekholm, D. Tonkonog,
\emph{Refined disk potentials for immersed Lagrangian surfaces},
J. Differential Geom. {\bf 121} (2022) 459--539.

\bibitem{EL}
T. Ekholm, Y. Lekili,
\emph{Duality between Lagrangian and Legendrian invariants},
Ekholm, Tobias; Lekili, Yankı
Duality between Lagrangian and Legendrian invariants.
Geom. Topol. {\bf 27} (2023) 2049--2179

\bibitem{ENS}
T. Ekholm, L. Ng, V. Shende, 
\emph{A complete knot invariant from contact homology},
Invent. Math. {\bf 211} (2018) 1149--1200.


\bibitem{EGH} 
Y. Eliashberg, A. Givental, H. Hofer,
	\emph{Introduction to symplectic field theory},
	Geom. Funct. Anal., Special Volume, Part II (2000), 560--673.


\bibitem{Ecurves} T. Ekholm,
\emph{Holomorphic curves for Legendrian surgery},
arXiv:1906.07228

\bibitem{G}
S. Gantara,
\emph{Symplectic Cohomology and Duality for the Wrapped Fukaya Category}
ProQuest LLC, Ann Arbor, MI, 2012.

\bibitem{GPS}
S. Ganatra, J. Pardon, V. Shende, 
\emph{Sectorial descent for wrapped Fukaya categories},
J. Amer. Math. Soc. {\bf 37} (2024) 499--635.

\end{thebibliography}
\end{document}